\documentclass[preprint,review,10pt]{amsart}

\newtheorem{theorem}{Theorem}[section]
\newtheorem{lemma}[theorem]{Lemma}

\theoremstyle{definition}

\newtheorem{example}[theorem]{Example}

\newtheorem{proposition}[theorem]{Proposition}
\newtheorem{corollary}[theorem]{Corollary}

\theoremstyle{remark}

\numberwithin{equation}{section}

\begin{document}

\title{ Eigenvaluues monotonicity of Witten-Laplacian along the  mean curvature flow}

\author{Shahroud Azami}
\address{Department of Mathematics, Faculty of Sciences,
Imam Khomeini International University, Qazvin, Iran.
}

\email{azami@sci.ikiu.ac.ir,\,\, shahrood78@yahoo.com }





\begin{abstract}
In this paper, we derive  the evolution equation for the first eigenvalue  of the Witten-Laplace operator acting on the space of functions along the  mean curvature flow on a closed oriented  manifold. We show some interesting monotonic quantities under  the mean curvature flow.
\end{abstract}
\maketitle
\section{Introduction}
In this paper we study  monotonicity of the first eigenvalue of Witten-Laplace operator  along the  mean curvature flow.
Over the last few
years  mean curvature flow and other geometric flows  have been a topic of active
research interest in both mathematics and physics and many mathematicians consider the problem under various   geometric flows. There are many results on the evolution and  monotonicity of eigenvalues of the geometric operator,  for instance $p$-Lapalcian and Witten-Laplacian, on evolving manifolds with or without curvature assumptions.\\
Notice that the main study of the properties of eigenvalues of the geometric operator, especially $p$-Laplacian, on evolving closed manifolds is still very young and began when
 Perelman \cite{GP} showed that the functional
\begin{equation*}
F=\int_{M} (R+|\nabla f|^{2}) e^{-f}\,d\nu
\end{equation*}
is increasing under the Ricci flow coupled to a backward heat-type equation, where $R$ is the scalar curvature with respect to the metric  $g(t)$ and $d\nu$ denotes the volume form  of the metric  $g=g(t)$.
The nondecreasing of the functional $F$ yields that the first  eigenvalue of the geometric  operator $-4\Delta+R$ is nondecreasing along the Ricci flow.
Then, Li \cite{JFL} and Cao \cite{XC2} extended the geometric operator $-4\Delta+R$ to the operator $-\Delta+cR$ and both them proved that the first eigenvalue of  the geometric  operator $-\Delta+cR$ for $c\geq\frac{1}{4}$ is nondecreasing along the Ricci flow.
Zeng and et 'al \cite{FZ} studied the monotonicity of eigenvalues of the operator $-\Delta+cR$ along the Ricci-Bourguignon flow. \\
On the other hand in    \cite{AA, JY,  LZ} has been investigated the evolution for the first eigenvalue of $p$-Laplacian along the Ricci-harmonic flow, Ricci flow and mth mean curvature flow, respectively.
 Mao \cite{JM}, studied the monotonicity of the first eigenvalue of the Laplace and the p-Laplace operators under a forced mean curvature flow.
 Author in  \cite{SA} shown that the first eigenvalue of  Witten-Laplace operator $-\Delta_{\phi}$ is monotonic along the Ricci-Bourguignon flow with some assumption where  $\phi\in C^{2}(M)$.
In  \cite{SF} and \cite{FSW} have been  studied the evolution for the first eigenvalue of geometric operator $-\Delta_{\phi}+\frac{R}{2}$ under the Yamabe flow and Ricci flow, respectively, and  constructed some monotonic quantities under this flow.\\


The mean curvature flow is a kind of heat flow equation for the immersion and   occurs in the description of the interface evolution in certain physical models. This is related to the property that such a flow  is the gradient flow of the area functional and therefore appears naturally in problems where a surface energy is minimized \cite{GH, GHA, CES}. Motivated by the above works, in this paper we will study the first eigenvalue of the Witten-Laplacian operator  whose metric satisfies  the  mean curvature flow.

\section{Preliminaries}
In this section, we recall some basic knowledge about mean curvature flow.
Let $M$ be a closed oriented $n$-dimensional manifold with  $n\geq2$ and  $F_{0}:M\to \mathbb{R}^{n+1}$
be a smooth immersion of $M$ into the Euclidean space. The evolution of $M_{0}=F_{0}(M)$ by unnormalized mean curvature flow is a one-parameter family of immersions $F:M\times [0,T)\to \mathbb{R}^{n+1}$ which satisfy the partial differential equation
\begin{equation}\label{mcf1}
\frac{\partial F}{\partial t}(x,t)=-H(x,t)\nu (x,t), \,\,\,\,\,\,\, F(.\, ,0)=F_{0},\,\,\,\,\,x\in M,\,\,\,t\geq 0,
\end{equation}
where $H(x,t)$ and $\nu(x,t)$ are the mean curvature and the normal at the point $F(x,t)$ of the surface $M_{t}=F(.\, ,t)(M)$, respectively.

From (\ref{mcf1}), we can conclude  that $-H(x,t)\nu (x,t)=\Delta_{M_{t}}F(x,t)$, where $\Delta_{M_{t}}$ is the Laplace-Beltrami operator on $M_{t}$. Thus, the  unnormalized mean curvature flow is a kind of heat equation for the immersion,  a  parabolic problem and it has unique solution for small times. Also,  we can apply  the maximum principle for the mean curvature flow.
The  volume-preserving mean curvature flow defined as
\begin{equation}\label{mcf2}
\frac{\partial F}{\partial t}(x,t)=(r(t)-H(x,t))\nu (x,t), \,\,\,\,\,\,\, F(.\, ,0)=F_{0},\,\,\,\,\,x\in M,\,\,\,t\geq 0,
\end{equation}
where  $r(t)=\frac{\int_{M} H d\mu_{t} } {\int_{M}d\mu_{t}}$ is a function of $t$. Short time
exitance and uniqueness for
solution to the mean curvature flow on $[0,T)$ have been shown   in  ~\cite{MRCS}. An extended of mean  curvature flow defined as
\begin{equation}\label{mcf221}
\frac{\partial F}{\partial t}(x,t)=-H^{k}(x,t)\nu (x,t), \,\,\,\,\,\,\, F(.\, ,0)=F_{0},\,\,\,\,\,x\in M,\,\,\,t\geq 0,
\end{equation}
where $k>0$ and $H^{k}$ is $k$th power of the mean curvature. This evolution equation called unnormalized $H^{k}$- flow.
 We consider a  interesting  generalization of the mean curvature  flow as follows
\begin{equation}\label{mcf3}
\frac{\partial F}{\partial t}(x,t)=-\mathcal{S}(x,t)\nu (x,t), \,\,\,\,\,\,\, F(.\, ,0)=F_{0},\,\,\,\,\,x\in M,\,\,\,t\geq 0,
\end{equation}
where $\mathcal{S}(x,t)$ is a given symmetric function of the principal curvatures of $M_{t}$ at $x,\, t$. For instance, if $S$ equal $H$, $H-r(t)$ or $H^{k}$ then this flow is a unnormalized mean curvature flow, the  volume-preserving mean curvature flow, $H^{k}$-flow, respectively.   Ritor\'{e} and Sinestrari  \cite{MRCS}, shown that, if $k_{1},...,k_{n}$  the principal curvatures of $M$ and  $F_{0}:M\to \mathbb{R}^{n+1}$ be a smooth immersion of a closed $n$-dimensional manifold $M$ such that at any point $x\in M$,
\begin{equation*}
\frac{\partial\mathcal{S}}{\partial k_{i}}(k_{1}(x),...,k_{n}(x))>0,\,\,\,\, i=1,2,...n,
\end{equation*}
then equation (\ref{mcf3}) has unique smooth solution on some time interval $[0,T)$. Also, from \cite{MRCS} we have
\begin{lemma}\label{l1}
If $M_{t}$ evolves by (\ref{mcf3}), with $\mathcal{S}$ a symmetric function of the curvatures homogeneous of degree $\gamma>0$, then the geometric quantities associated to $M_{t}$ satisfy the following equations:
\begin{description}
  \item[i]$\frac{ \partial}{\partial t }g_{ij}=-2\mathcal{S}h_{ij}$,\,\,\,\,\,\,$\frac{ \partial}{\partial t }g^{ij}=2\mathcal{S}h^{ij}$,
  \item[ii] $\frac{ \partial}{\partial t }d\mu=-\mathcal{S}Hd\mu$,
  \item[iii] $\frac{ \partial}{\partial t }h_{ij}=\frac{ \partial \mathcal{S}}{\partial h_{l}^{m} }\nabla^{l}\nabla_{m}h_{ij}+
\frac{ \partial^{2} \mathcal{S}}{\partial h_{r}^{s} \partial h_{l}^{k} }\nabla_{i}h_{r}^{s}\nabla_{j}h_{l}^{k}
+\frac{ \partial \mathcal{S}}{\partial h_{l}^{k} }h_{m}^{k}h_{l}^{m}h_{ij}-(\gamma+1)h_{ik}h_{j}^{k}\mathcal{S}$,
  \item[iv] $\frac{ \partial}{\partial t }H=\frac{ \partial \mathcal{S}}{\partial h_{i}^{j} }\nabla^{i}\nabla_{j}H+
\frac{ \partial^{2} \mathcal{S}}{\partial h_{r}^{s} \partial h_{l}^{k} }\nabla^{i}h_{r}^{s}\nabla_{i}h_{l}^{k}
+\frac{ \partial \mathcal{S}}{\partial h_{l}^{k} }h_{m}^{k}h_{l}^{m}H-(\gamma-1)|A|^{2}\mathcal{S}$,
\end{description}
where $A=(h_{ij})$,  $h_{ij}=<\partial_{i}F, \partial_{j}\nu>_{e}$ is the second fundamental form associated with $F:M\to \mathbb{R}^{n+1}$ and $<.,.>_{e}$ denotes the Euclidean scalar product in $\mathbb{R}^{n+1}$.
\end{lemma}
\subsection{Eigenvalues of Witten-Laplace operator}
In this section, we will first  give the definitions for the first eigenvalue of the Witten-Laplace  operator
$\Delta_{\phi}$
 then we will find the formula for the evolution of the first eigenvalue of the Witten-Laplace operator     under the  evolution equation (\ref{mcf3}) on a closed oriented manifold.
Let $M$ be a closed oriented Riemannian manifold, and $(M, g(t))$ be a smooth solution of  the evolution equation (\ref{mcf3})  for $t\in[0,T)$.
Let $\nabla$ be the Levi-Civita  connection on $(M, g(t))$  and  $f:M\to \mathbb{R}$ be a smooth function on $M$ or $f\in W^{1,2}(M)$  where $W^{1,2}(M)$ is  the Sobolev space. The Laplacian of $f$ is defined as
\begin{equation}\label{l}
\Delta f=div(\nabla f)=g^{ij}(\partial_{i}\partial_{j} f-\Gamma_{ij}^{k}\partial_{k}f).
\end{equation}
Let  $d\nu$  be the Riemannian volume measure, and $d\mu$ be  the weight volume measure on  $(M, g(t))$ related to function $\phi$; i.e.
\begin{equation}\label{l1}
d\mu=e^{-\phi(x)}d\nu
\end{equation}
where $\phi \in C^{2}(M)$. The Witten-Laplacian is defined by
\begin{equation}\label{wl}
\Delta_{\phi}=\Delta-\nabla\phi.\nabla
\end{equation}
which  is a symmetric operator on $L^{2}(M,\mu)$ and satisfies  the following  integration by part formula:
\begin{equation*}
\int_{M}<\nabla u, \nabla v>d\mu=-\int_{M} v\Delta_{\phi}u\,d\mu=-\int_{M}u \Delta_{\phi}v\, d\mu\qquad\qquad \forall u,v\in C^{\infty}(M),
\end{equation*}
The Witten-Laplacian is generalize of Laplacian operator, for instance, when $\phi$ is a constant function, the Witten-Laplacian  operator is just the Laplace-Belterami operator. \\
We say that  $\lambda_{1}(t)$
is an eigenvalue of the Witten-Laplace  operator $\Delta_{\phi}$ at time $t\in[0,T)$ whenever for some $f\in W^{1,2}(M)$,
\begin{equation}\label{e1}
-\Delta_{\phi}f=\lambda_{1}(t)f,
\end{equation}
or equivalently
\begin{equation}\label{cdc1}
\int_{M}<\nabla f, \nabla h>d\mu=\lambda_{1} \int_{M} f\,h\,d\mu,\qquad\qquad \forall h\in C^{\infty}(M).
\end{equation}
If in above formula we set $h=f$ then
\begin{equation*}
\lambda_{1}=\frac{\int_{M}|\nabla f|^{2}d\mu}{\int_{M} f^{2}\,d\mu}.
\end{equation*}
therefore the first eigenvalue of the Witten-Laplace  operator defined as
\begin{equation*}
\lambda=\mathop{\min}\limits_{f\neq0}\left \{\int_{M}|\nabla f|^{2}d\mu : \,\, f\in C^{\infty}(M),\,\,\,\int_{M}f^{2}d\mu=1\right\}.
\end{equation*}
Author \cite{SA} shown that the following Lemma:
\begin{lemma}
If $g_{1}$ and $g_{2}$ are two metrics on  Riemannian manifold $M$  which satisfy
\begin{equation*}
\frac{1}{1+\epsilon}g_{1}\leq g_{2}\leq(1+\epsilon )g_{1},
\end{equation*}
then
\begin{equation*}
\lambda(g_{2})-\lambda(g_{1})\leq\left( (1+\epsilon )^{\frac{n}{2}+1}-(1+\epsilon )^{-\frac{n}{2}} \right)(1+\epsilon)^{\frac{n}{2}}\lambda(g_{1}).
\end{equation*}
In particular, $\lambda $ is a continuous function respect to the $C^{2}$-topology.
\end{lemma}
we say that $f$ is normalized eigenfunction corresponding to $\lambda$ whenever  $\lambda=\int_{M}|\nabla f|^{2} d\mu$ and $\int_{M}f^{2}d\mu=1$.
Similar method of \cite{SA}, at  time  $t_{0}\in[0,T)$, we first let $f_{0}=f(t_{0})$ be the eigenfunction for the eigenvalue $\lambda(t_{0})$ of Witten-Laplacian. We consider the following smooth function
$$l(t)=f_{0}\left[\frac{\det(g_{ij}(t_{0}))}{\det(g_{ij}(t))}\right]^{\frac{1}{2}}$$
along the  evolution equation (\ref{mcf3}). We assume that
$$f(t)=\frac{l(t)}{\left(\int_{M}(l(t))^{2}d\mu\right)^{\frac{1}{2}}}$$
which $f(t)$ is smooth function under the  evolution equation (\ref{mcf3}), satisfies $\int_{M}f^{2}d\mu=1$ and at time $t_{0}$, $f$ is the eigenfunction for $\lambda$ of Witten-Laplacian. Now we define a smooth eigenvalue function
\begin{equation}\label{R6}
\lambda(f,t):=\int_{M}|\nabla f|^{2}d\mu
\end{equation}
where $\lambda(f(t_{0}),t_{0})=\lambda(t_{0}),\,\, f$ is smooth function and satisfies
\begin{equation}\label{e2}
\int_{M} f^{2}\,d\mu=1.
\end{equation}
\section{Variation of $\lambda(t)$}
In this section, we will give some useful evolution formulas for $\lambda(t)$ under the evolution equation (\ref{mcf3}). Now, we give a useful proposition  about the variation of eigenvalues of  Witten-Laplacian under the  evolution equation (\ref{mcf3}).
\begin{proposition}\label{pro1}
Let  $(M^{n}, g(t))$ be  a solution of  the   evolution equation (\ref{mcf3})
on the smooth  closed manifold $(M^{n}, g_{0})$. If $\lambda(t)$ denotes the
evolution of the first  eigenvalue under the evolution equation (\ref{mcf3}), then
\begin{equation}\label{R7}
\frac{d}{dt}\lambda(f,t)|_{t=t_{0}}=\lambda(t_{0})\int_{M}\mathcal{S}H f^{2}d\mu
+2\int_{M}\mathcal{S}h^{ij}\nabla_{i}f\nabla_{j} fd\mu-\int_{M}|\nabla f|^{2}\mathcal{S}H d\mu
\end{equation}
\end{proposition}
\begin{proof}
According to the above description,
 $ \lambda(f,t)$ is a smooth function and by  derivation  (\ref{R6}) we have
\begin{equation}\label{R8}
\frac{d}{dt}\lambda(f,t)=\int_{M}\frac{d}{dt}(|\nabla
f|^{2})d\mu+\int_{M}|\nabla f|^{2}\frac{d}{dt}(d\mu).
\end{equation}
On the other hand, we have
\begin{equation}\label{R9}
\frac{d}{dt}(d\mu_{t})=-\mathcal{S}Hd\mu
\end{equation}
and
\begin{eqnarray}\label{R10}
\frac{d}{dt}\bigg(|\nabla f|^{2}\bigg)&=&\frac{d}{dt}\bigg(g^{ij}\nabla_{i}f\nabla_{j}f\bigg)=\frac{\partial}{\partial t}(g^{ij})\nabla_{i}f
\nabla_{j}f+2g^{ij}\nabla_{i}f'\nabla_{j}f\\\nonumber
&=&-g^{il}g^{jk}\frac{\partial}{\partial
t}(g_{lk})\nabla_{i}f\nabla_{j}f+2<\nabla f',\nabla f>.\nonumber
\end{eqnarray}
Plug in  (\ref{R9}),  (\ref{mcf3}) and (\ref{R10}) into (\ref{R8}), yields
\begin{eqnarray}\label{R11}
\frac{d}{dt}\lambda(f,t)=2\int_{M}\left\{\mathcal{S}h^{ij}\nabla_{i}f\nabla_{j}f+<\nabla f',\nabla f>\right\}d\mu-\int_{M}|\nabla f|^{2}\mathcal{S}Hd\mu.
\end{eqnarray}
Now, using (\ref{e2}), from the condition
\begin{equation*}
 \int_{M}f^{2}d\mu=1,
\end{equation*}
and the time derivative, we  can get
\begin{equation}\label{R14}
 2\int_{M}f'fd\mu=\int_{M}f^{2}\mathcal{S} H d\mu,
\end{equation}
(\ref{cdc1}) and (\ref{R14}) imply that at time $t=t_{0}$, we have
\begin{equation}\label{R15}
\int_{M}<\nabla f',\nabla f>d\mu=\lambda(t_{0})
\int_{M}f'fd\mu=\frac{\lambda(t_{0})}{2}\int_{M}f^{2}\mathcal{S} H d\mu.
\end{equation}
Replacing   (\ref{R15}) in (\ref{R11}), we obtain (\ref{R7}).
\end{proof}
Now, if  in  the pervious proposition we suppose that $\mathcal{S}=H$, $\mathcal{S}=H-r$  or  $\mathcal{S}=H^{k}$ then  we can obtain variation of $\lambda(t)$ under the volume-preserving,  unnormalized mean curvature flow and  $H^{k}$-flow  as follow:
\begin{corollary}
Let  $(M^{n}, g(t))$ be  a solution of  the  unnormalized mean curvature flow
on the smooth  closed oriented  manifold $(M^{n}, g_{0})$. If $\lambda(t)$ denotes the
evolution of a first eigenvalue under the unnormalized  mean curvature  flow, then
\begin{equation}\label{R16}
\frac{d\lambda}{dt}(f,t)|_{t=t_{0}}=\lambda(t_{0})\int_{M}H^{2}f^{2}d\mu
+2\int_{M}H h^{ij}\nabla_{i} f\nabla_{j} fd\mu-\int_{M}|\nabla f|^{p}H^{2} d\mu.
\end{equation}
Also, for  the volume-preserving  mean curvature  flow, we have
\begin{eqnarray}\label{R17}\nonumber
\qquad\frac{d\lambda}{dt}(f,t)|_{t=t_{0}}&=&\lambda(t_{0})\int_{M}(H-r)H f^{2}d\mu
+2\int_{M}(H-r)h^{ij}\nabla_{i} f\nabla_{j} f d\mu\\
&&-\int_{M}|\nabla f|^{2}(H-r)H d\mu.
\end{eqnarray}
and for unnormalized $H^{k}$-flow, we obtain
\begin{equation}\label{R162}
\qquad\frac{d\lambda}{dt}(f,t)|_{t=t_{0}}=\lambda(t_{0})\int_{M}H^{k+1}f^{2}d\mu
+2\int_{M}H^{k} h^{ij}\nabla_{i} f\nabla_{j} fd\mu-\int_{M}|\nabla f|^{p}H^{k+1} d\mu.
\end{equation}
\end{corollary}
In the following theorem, we show that the evolution of the first eigenvalue (\ref{R6}) under the unnormalized mean curvature flow with some  conditions at the initial time, is  nondecreasing.
\begin{theorem}\label{tt1}
Let  $(M^{n}, g(t))$ be  a solution of  the  unnormalized mean curvature flow on the smooth  closed manifold $(M^{n}, g_{0})$ and  $\lambda(t)$ denotes the evolution of the first eigenvalue under the unnormalized mean curvature flow. At the initial time $t=0$, if $H>0$, and there exists a non-negative constant $\epsilon$ such that
\begin{equation}\label{t1}
h_{ij}\geq\epsilon Hg_{ij},\,\,\,\,\, \frac{1}{2}\leq \epsilon
\end{equation}
 then $\lambda(t)$ is nondecreasing  under  the  unnormalized mean curvature flow.
\end{theorem}
\begin{proof}G.  Huisken  in ~\cite{GH},
 by the weak maximum principle for symmetric tensors  shown that
\begin{equation*}
h_{ij}\geq\epsilon Hg_{ij}\,\,\,\, \,\text{in} \,\, M^{n}\times[0,T),
\end{equation*}
thus $h_{ij}\geq\epsilon Hg_{ij} $ is preserved by the unnormalized mean curvature flow which implies that
\begin{equation*}
h^{ij}\nabla_{i}f\nabla_{j}f \geq\epsilon H|\nabla f|^{2}\,\,\,\, \,\text{in} \,\, M^{n}\times[0,T).
\end{equation*}
Replacing  this into (\ref{R16}) yields
\begin{eqnarray*}\nonumber
\frac{d\lambda}{dt}(f,t)|_{t=t_{0}}
&\geq&\lambda(t_{0})\int_{M}H^{2} f^{2}d\mu+(2\epsilon-1)\int_{M}|\nabla f|^{2}H^{2}d\mu
\end{eqnarray*}
we arrive at $\frac{d}{dt}\lambda(f(t),t)>0$  in any sufficiently small neighborhood of $t_{0}$, then
\begin{equation*}
\lambda(f(t_{1}),t_{1})<\lambda(f(t_{0}),t_{0})\,\,\,\,\,\, \text{on}\,\,\,\,\,\,  [t_{1},t_{0}].
\end{equation*}
 Since $\lambda(f(t_{0}),t_{0})=\lambda(t_{0})$ and  $\lambda(f(t_{1}),t_{1})\geq\lambda(t_{1})$ we conclude that $\lambda(t_{1})<\lambda(t_{0})$  which show that $\lambda(t)$ is strictly monotone increasing
 in any sufficiently small neighborhood of $t_{0}$.  Then $\lambda(t)$ is strictly increasing along the evolution equation (\ref{mcf3})  on $[0,T)$ because of  $t_{0}$ is arbitrary.
\end{proof}
Now, suppose that   the  evolution equation (\ref{mcf3}),  with $\mathcal{S}=(H^{2}- \tilde{r}(t))$ where $\tilde{r}(t)=\frac{\int_{M}H^{2}d\mu}{\int_{M}d\mu}$, i.e.
\begin{equation}\label{vmcf1}
\frac{\partial F}{\partial t}(x,t)=(\tilde{r}(t)-H^{2}(x,t))\nu (x,t), \,\,\,\,\,\,\, F(.\, ,0)=F_{0},\,\,\,\,\,x\in M,\,\,\,t\geq 0.
\end{equation}
In the next theorem we study some  monotonic quantities dependent  to   the first eigenvalue  (\ref{R6})  under the equation (\ref{vmcf1}).
\begin{theorem}
Let  $(M^{n}, g(t))$ be  a solution of  the evolution  equation (\ref{vmcf1}) on the smooth  closed oriented manifold $(M^{n}, g_{0})$ and  $\lambda(t)$ denotes the evolution of the first eigenvalue  (\ref{R6}) under   the evolution  equation (\ref{vmcf1}). If at the initial time $t=0$, $H>0$ and there exist non-negative constant $\epsilon$ such that $
h_{ij}\geq\epsilon Hg_{ij}
$
then there exist $\varphi(t)$   and $\psi(t)$ such that the quantities
\begin{equation*}
e^{-\int_{0}^{t}\left( \psi(\tau)-\varphi(\tau)+2\epsilon \psi(\tau) \right)d\tau}\lambda(t)
\end{equation*}
is nondecreasing along the evolution  equation  (\ref{vmcf1}). Similarly,
\begin{equation*}
e^{-\int_{0}^{t}\left( \varphi(\tau)-\psi(\tau)+2 \varphi(\tau) \right)d\tau}\lambda(t)
\end{equation*}
 is nonincreasing along  the evolution  equation  (\ref{vmcf1}).
\end{theorem}
\begin{proof}
Rivas and Sinestrari in  \cite{CES} shown that the inequality $\epsilon H g_{ij}\leq h_{ij}\leq H g_{ij}$ preserves along the  evolution  equation  (\ref{vmcf1}), therefore
\begin{equation*}
\epsilon H |\nabla f|^{2}  \leq
h^{ij}\nabla_{i}f\nabla_{j}f \leq H|\nabla f|^{2}\,\,\,\, \,\text{in} \,\, M^{n}\times[0,T).
\end{equation*}
 Moreover, since $H>0$ in $t=0$  by the maximum principle they shown the following inequalities also preserve.
\begin{equation}
0\leq C_{1}e^{-C_{2}t}\leq H\leq C_{3},
\end{equation}
for some constants  $C_{1}, C_{2}$ and $ C_{3}$.  Therefore,
 \begin{equation*}
C_{1}^{2}e^{-2C_{2}t}-\tilde{r}(t)\leq \mathcal{S}=(H^{2}- \tilde{r}(t)) \leq C_{3}^{2}-\tilde{r}(t)
\end{equation*}
and
 \begin{equation*}
\psi(t)=(C_{1}^{2}e^{-2C_{2}t}-\tilde{r}(t))C_{1}e^{-C_{2}t} \leq \mathcal{S}H=(H^{2}- \tilde{r}(t))H \leq (C_{3}^{2}-\tilde{r}(t))C_{3}=\varphi(t).
\end{equation*}
Now (\ref{R7}) results  that
\begin{eqnarray*}
\frac{d\lambda}{dt}(f,t)|_{t=t_{0}}
&\geq&\lambda(t_{0})\int_{M}\mathcal{S}H f^{2}d\mu
+2\epsilon \int_{M}\mathcal{S}H|\nabla f|^{2}d\mu
-\int_{M}|\nabla f|^{2}\mathcal{S}H d\mu\\\nonumber
&\geq& \lambda(t_{0}) \psi(t_{0})+2\epsilon  \psi(t_{0}) \int_{M}|\nabla f|^{2}d\mu
- \varphi(t)\int_{M}|\nabla f|^{2} d\mu\\\nonumber
&\geq& \lambda(t_{0})\left( \psi(t_{0})-\varphi(t_{0})+2\epsilon \psi(t_{0}) \right)
\end{eqnarray*}
this results that in any enough small neighborhood of $t_{0}$ as $I_{0}$,  we have
$$\frac{d}{dt}\lambda(f,t)\geq  \lambda(t)\left( \psi(t)-\varphi(t)+2\epsilon \psi(t) \right). $$
Integrating the last inequality with respect to variable time $t$ on $[t_{1}, t_{0}]\subset I_{0}$, we get
\begin{equation*}
\ln \frac{\lambda(f(t_{0}),t_{0})}{ \lambda(f(t_{1}),t_{1})}>\int_{t_{1}}^{t_{0}}\left( \psi(\tau)-\varphi(\tau)+2\epsilon \psi(\tau) \right)d\tau.
\end{equation*}
Since $\lambda(f(t_{0}),t_{0})=\lambda(t_{0})$ and  $\lambda(f(t_{1}),t_{1})\geq\lambda(t_{1})$ we conclude that
\begin{equation*}
\ln \frac{\lambda(t_{0})}{ \lambda(t_{1})}>\int_{t_{1}}^{t_{0}}\left( \psi(\tau)-\varphi(\tau)+2\epsilon \psi(\tau) \right)d\tau,
\end{equation*}
 that is the quantity  $e^{-\int_{0}^{t}\left( \psi(\tau)-\varphi(\tau)+2\epsilon \psi(\tau) \right)d\tau}\lambda(t)$ is strictly increasing  in any sufficiently small neighborhood of $t_{0}$. Since $t_{0}$ is arbitrary, then
\begin{equation*}
e^{-\int_{0}^{t}\left( \psi(\tau)-\varphi(\tau)+2\epsilon \psi(\tau) \right)d\tau}\lambda(t)
\end{equation*}
is nondecreasing along  the mean curvature flow (\ref{vmcf1}) on $[0,T)$. Similarly,
\begin{equation*}
e^{-\int_{0}^{t}\left( \varphi(\tau)-\psi(\tau)+2\varphi(\tau) \right)d\tau}\lambda(t)
\end{equation*}
is nonincreasing along the mean curvature flow (\ref{vmcf1}).
\end{proof}
In the following, we need to use Hamilton's  maximum principle for tensors on manifolds (see \cite{RSH}) which is
\begin{theorem}[\cite{RSH}]
Suppose that on $[0,T)$ the evolution equation
\begin{equation*}
\frac{\partial}{\partial t}M_{ij}=\Delta M_{ij}+u^{k}\nabla_{k}M_{ij}+N_{ij}
\end{equation*}
holds, where $N_{ij}=P(M_{ij}, g_{ij})$, a polynomial in $M_{ij}$ formed by contracting products of $M_{ij}$ with itself using  the metric, satisfies $N_{ij}X^{i}X^{j}\geq0$ for any null eigenvector  $X=\{X^{i}\}$ of  $M_{ij}$. If  $M_{ij}\geq 0$ at $t=0$ then it remains so on  $[0,T)$.
\end{theorem}
\begin{lemma}\label{ll1}
Let $M_{t}^{n}$ be a solution of unnormalized $H^{k}$-flow on the smooth  closed oriented  manifold $M_{0}^{n}$.
If  there exist positive constant $a_{1}, a_{2},...,a_{n}$ such that the initial hypersurface $M_{0}^{n}$ satisfies
\begin{equation}\label{mm}
h_{ij}=a_{i}Hg_{ij},\,\,\,\,\,\,\text{where} \,\, \sum_{i=1}^{n}a_{i}=1,\,\,\text{and}\,\,\,\,\,\,\,0\leq \frac{1}{n}-a_{i}\leq \epsilon
\end{equation}
for small enough $\epsilon$ only depending on $n$,  then (\ref{mm}) remains  along the  unnormalized $H^{k}$-flow for any $t\in[0,T)$.
\end{lemma}
\begin{proof}
By  (\ref{mm}), we have
\begin{equation*}
a_{i}Hg_{ij}\leq h_{ij}\leq a_{i}Hg_{ij},\,\,\,\,\,\,\text{on} \,\,\,M_{0}.
\end{equation*}
On the other hand, by Lemma (\ref{l1}) for $\mathcal{S}=H^{k}$ we get
\begin{equation}
\frac{\partial H}{\partial t}=kH^{k-1}\Delta H+k(k-1)H^{k-2}|\nabla H|^{2}+|A|^{2}H^{k}
\end{equation}
and
\begin{equation}
\frac{\partial h_{ij}}{\partial t}=-(k+1)H^{k}h_{il}h_{j}^{l}+kH^{k-1}\Delta h_{ij}+k(k-1)H^{k-2}\nabla_{i}H\nabla_{j}H+kH^{k-1}|A|^{2}h_{ij}.
\end{equation}
 From the classical maximum principle, we conclude that $H^{k}\geq 0$ for all  $k\in\mathbb{N}$ and $t\in[0,T)$.
If  $M_{ij}=h_{ij}-\frac{1}{2}a_{i}Hg_{ij}$ then
\begin{eqnarray}\nonumber
\frac{\partial M_{ij}}{\partial t}&=& kH^{k-1}\Delta M_{ij}+k(k-1)H^{k-2}\nabla_{i}H\nabla_{j}H-\frac{1}{2}a_{i}k(k-1)H^{k-1}|\nabla H|^{2}g_{ij}\\&&
-(k+1)H^{k}h_{il}h_{j}^{l}-\frac{1}{2}a_{i}|A|^{2}H^{k}g_{ij}+a_{i}H^{k+1}h_{ij}+kH^{k-1}|A|^{2}h_{ij}.
\end{eqnarray}
Suppose that
\begin{equation}
N_{ij}=-(k+1)H^{k}h_{il}h_{j}^{l}-\frac{1}{2}a_{i}|A|^{2}H^{k}g_{ij}+a_{i}H^{k+1}h_{ij}+kH^{k-1}|A|^{2}h_{ij}.
\end{equation}
By assumption of lemma  $M_{ij}\geq 0$ at $t=0$. We now show that $N_{ij}$ is nonnegative  on the  null-eigenvectors of $M_{ij}$. Assume that, For some vector $X=\{X^{i}\}$, we have $h_{ij}X^{j}=\frac{1}{2}a_{i}HX_{i}$. Thus, we get
\begin{eqnarray}\nonumber
N_{ij}X^{i}X^{j}&=&-\frac{1}{4}(k+1)H^{k+2}a_{i}a_{j}g^{ij}X_{i}X_{j}-\frac{1}{2}a_{i}|A|^{2}H^{k}g_{ij}X^{i}X^{j}+\frac{1}{2}a_{i}^{2}H^{k+2}X^{i}X_{i}\\\nonumber&&
+\frac{1}{2}ka_{i}H^{k}|A|^{2}X^{i}X_{i}\\\nonumber
&=&\frac{1}{4}H^{k+2}(-(k+1)a_{i}a_{j}g^{ij}X_{i}X_{j}+2a_{i}^{2}X^{i}X_{i})\\\nonumber
&&+\frac{1}{2}H^{k}|A|^{2}(-a_{i}g_{ij}X^{i}X^{j}+ka_{i}X^{i}X_{i})\\\nonumber
&=&\frac{1}{4}(k-1)H^{k}a_{i}(2|A|^{2}-a_{i}H^{2})X_{i}X^{i}.
\end{eqnarray}
Since $|A|^{2}\geq \frac {H^{2}}{n}$, we have $N_{ij}X^{i}X^{j}\geq 0$. Not that
\begin{equation}
\nabla H=\nabla (h_{ij}g^{ij})=\nabla [(M_{ij}+\frac{1}{2}a_{i}Hg_{ij})g^{ij}]=g^{ij}\nabla M_{ij}+\frac{1}{2}\nabla H
\end{equation}
which implies $\nabla H=2 g^{ij}\nabla M_{ij}$. Now, Hamilton's maximum principle to tensor implies that $M_{ij}\geq 0$ on $M_{t}$ for any $t\in [0,T)$,  that is  $a_{i}Hg_{ij}\leq h_{ij}$.  Similarly,  on $M_{t}$ we have $a_{i}Hg_{ij}\geq h_{ij}$ for any  $t\in [0,T)$. So we get $h_{ij}=a_{i}Hg_{ij}$ on $M\times [0,T)$.
\end{proof}
\begin{theorem}
Let  $(M^{n}, g(t))$ be  a solution of  the  unnormalized $H^{k}$-flow
on the smooth  closed oriented  manifold $(M^{n}, g_{0})$ and  $\lambda(t)$ denotes the
evolution of a first eigenvalue under the unnormalized  mean curvature  flow. If   initial hypersurface $M_{0}$ satisfies (\ref{mm}) then under the unnormalized m $H^{k}$-flow $\lambda(t)$ is nondecreasing.
\end{theorem}
\begin{proof}
From  Lemma \ref{ll1}  we have
\begin{equation}
h_{ij}=a_{i}Hg_{ij},\,\,\,\,\,\,\text{on} \,\,\,M_{t}.
\end{equation}
So, from (\ref{R16}) we conclude that
\begin{eqnarray*}
\frac{d\lambda}{dt}(f,t)|_{t=t_{0}}
&=&\lambda(t_{0})\int_{M}H^{2} f^{2}d\mu
+2\int_{M}H^{2} a_{i}g^{ij}\nabla_{i} f \nabla_{j}fd\mu-\int_{M}|\nabla f|^{2}H^{2} d\mu\\
&\geq&\lambda(t_{0})\int_{M}H^{2} f^{2}d\mu
+2(\frac{1}{n}-\epsilon-\frac{1}{2})\int_{M}H^{2}|\nabla f|^{2}d\mu
\end{eqnarray*}
so for $\epsilon\leq \frac{1}{n}-\frac{1}{2}$, we have
 $\frac{d}{dt}\lambda(f(t),t)>0$  in any sufficiently small neighborhood of $t_{0}$, then
\begin{equation*}
\lambda(f(t_{1}),t_{1})<\lambda(f(t_{0}),t_{0})\,\,\,\,\,\, \text{on}\,\,\,\,\,\,  [t_{1},t_{0}].
\end{equation*}
 Since $\lambda(f(t_{0}),t_{0})=\lambda(t_{0})$ and  $\lambda(f(t_{1}),t_{1})\geq\lambda(t_{1})$ we conclude that $\lambda(t_{1})<\lambda(t_{0})$, hence $\lambda(t)$ is  nondecreasing
 in any sufficiently small neighborhood of $t_{0}$.  Then $\lambda(t)$ is nondecreasing along $H^{k}$-flow  on $[0,T)$ because of  $t_{0}$ is arbitrary.
\end{proof}
\begin{example}
Let  $M_{0}=\mathbb{S}_{R}^{n}(0)$  be the sphere of radius $R$ around the origin. The mean curvature $\mathbb{S}_{R}^{n}(0)$ is  equal to $\frac{n}{R}$. If we set $F(x,t)=r(t)F_{0}(x)$, where $F_{0}(x)$ is the standard immersion of $M_{0}$ in $\mathbb{R}^{n+1}$, then we have
\begin{equation*}
R'(t)F_{0}(x)=\frac{\partial}{\partial t}F(x,t)=-H(x,t)\nu(x,t)=-\frac{n}{r(t)}F_{0}(x),
\end{equation*}
therefore the evolution of $M_{0}$ is given by sphere $M_{t}=\mathbb{S}_{r(t)}^{n}(0)$, where the radius $r(t)$ evolves according to the ordinary differential equation
 \begin{equation*}
r'(t)=-\frac{n}{r(t)},\,\,\,\,\,r(0)=R.
\end{equation*}
The solution is given by $r(t)=\sqrt{R^{2}-2nt}$, thus
 \begin{equation*}
H(x,t)=\frac{n}{\sqrt{R^{2}-2nt}},\,\,\,\,\,h^{ij}(x,t)=\frac{1}{\sqrt{R^{2}-2nt}}g^{ij}(t)=\frac{H(x,t)}{n}g^{ij}(t).
\end{equation*}
which these and (\ref{R16}) imply that
\begin{equation}\label{e1}
\frac{d\lambda}{dt}(f,t)|_{t=t_{0}}
=2\frac{H^{2}(x,t_{0})}{n}\int_{M}|\nabla f|^{2}d\mu=2\frac{H^{2}(x,t_{0})}{n}\lambda(t_{0}).
\end{equation}
we arrive at $\frac{d}{dt}\lambda(f(t),t)>0$  in any sufficiently small neighborhood of $t_{0}$, therefore
\begin{equation*}
\lambda(f(t_{1}),t_{1})<\lambda(f(t_{0}),t_{0})\,\,\,\,\,\, \text{on}\,\,\,\,\,\,  [t_{1},t_{0}].
\end{equation*}
 Since $\lambda(f(t_{0}),t_{0})=\lambda(t_{0})$ and  $\lambda(f(t_{1}),t_{1})\geq\lambda(t_{1})$ we get $\lambda(t_{1})<\lambda(t_{0})$  which show that $\lambda(t)$ is strictly monotone increasing
 in any enough small neighborhood of $t_{0}$.  Then $\lambda(t)$ is strictly increasing along the mean curvature flow.
\end{example}

\end{document}